\documentclass[12pt]{amsart}
\usepackage{amsmath,amssymb,amsfonts,amscd}
\usepackage{graphicx}
\theoremstyle{plain}
\newtheorem{thm}{Theorem}
\newtheorem{lem}[thm]{Lemma}

\newtheorem{prop}[thm]{Proposition}
\newtheorem{remark}[thm]{Remark}

\numberwithin{equation}{section}
\numberwithin{thm}{section}
\newcommand{\eq}[2]{\begin{equation}\label{#1}#2 \end{equation}}
\newcommand{\ml}[2]{\begin{multline}\label{#1}#2 \end{multline}}
\newcommand{\ga}[2]{\begin{gather}\label{#1}#2 \end{gather}}

\newcommand{\surj}{\twoheadrightarrow}
\newcommand{\inj}{\hookrightarrow}

\newcommand{\sA}{{\mathcal A}}

\newcommand{\sO}{{\mathcal O}}

\newcommand{\sW}{{\mathcal W}}

\newcommand{\C}{{\Bbb C}}

\newcommand{\G}{{\Bbb G}}

\renewcommand{\P}{{\Bbb P}}
\newcommand{\Q}{{\Bbb Q}}
\newcommand{\R}{{\Bbb R}}

\newcommand{\U}{{\Bbb U}}

\newcommand{\Z}{{\Bbb Z}}

\newcommand{\eps}{{\varepsilon}}

\title{A Note on Twistor Integrals}
\author{Spencer Bloch}\address{5765 S. Blackstone Ave., Chicago, IL 60637,
USA\\
E-mail address: spencer\_bloch@yahoo.com}
\begin{document} 
\maketitle
\section{Introduction}

This paper is a brief introduction to twistor integrals from a mathematical point of view. It was inspired by a paper of Hodges \cite{H} which we studied in a seminar at Cal Tech directed by Matilde Marcoli. The idea is to write the amplitude for a graph with $n$ loops and $2n+2$ propagators using the geometry of pfaffians for sums of rank $2$ alternating matrices. (Hodges considers the case of $1$ loop and $4$ edges). Why is this of interest to a mathematician? The Feynman amplitude is a {\it period} in the sense of arithmetic algebraic geometry. In parametric form, the amplitude integral associated to a graph $\Gamma$ with $N$ edges and $n$ loops has the form
\eq{1.1c}{c(N,n)\int_\delta \frac{S_1^{N-2n -2}\Omega}{S_2^{N-2n}}.
}
Here $S_1$ and $S_2$ are the first and second Symanzik polynomials \cite{BK},  \cite{BEK}, \cite{IZ},  and $\Omega = \sum \pm A_idA_1\wedge\cdots\wedge \widehat{dA_i}\wedge\cdots\wedge dA_N$ is the integration form on $\P^{N-1}$, the projective space with homogeneous coordinates indexed by edges of $\Gamma$. The chain of integration $\delta$ is the locus of points on $\P^{N-1}$ where all the $A_i\ge 0$. Note $\Omega,\ S_1,\ S_2$ are homogeneous of degrees $N,\ n, \ n+1$ in the $A_i$, so the integrand is homogeneous of degree $0$ and represents a rational differential form. Finally, $c(N,n)$ is some elementary constant depending only on $N$ and $n$. 

Two special cases suggest themselves. In the {\it log divergent} case when $N=2n$, the integrand is simply $\Omega/S_1^2$. The first Symanzik polynomial depends only on the edge variables $A_i$, so the result in this case is a constant. (If the graph is non-primitive, i.e. has log divergent subgraphs, the integral will diverge. We do not discuss this case.) Inspired by the conjectures of Broadhurst and Kreimer \cite{BrK}, there has been a great deal of work done on the primitive log divergent amplitudes. 

The polynomial $S_1$ itself is the determinant of an $n\times n$-symmetric matrix with entries linear forms in the $A_i$. The linear geometry of this determinant throws an interesting light on the motive of the hypersurface $X(\Gamma):S_1=0$. For example, one has a ``Riemann-Kempf'' style theorem that the dimension of the null space of the matrix at a point is equal to the multiplicity of the point on $X(\Gamma)$, \cite{P}, \cite{K}. Furthermore, the projectivized fibre space $Y(\Gamma)$ of these null lines maps birationally onto $X(\Gamma)$ and in some sense ``resolves'' the motive. Whereas the motive of $X(\Gamma)$ can be quite subtle, the motive of $Y(\Gamma)$ is quite elementary. In particular, it is mixed Tate \cite{B}. (The Riemann-Kempf theorem refers to the map $\pi: Sym^{g-1}C \surj \Theta \subset J_{g-1}(C)$ where $C$ is a Riemann surface and $\Theta$ is the theta divisor. The dimension of the fibre of $\pi$ at a point of $\Theta$ equals the multiplicity of the divisor $\Theta$ at the point.) 

The second case is $N=2n+2$, e.g. one loop and $4$ edges. The amplitude is $\int_\delta\Omega/S_2^2$ and is a function of external momenta and masses. The second Symanzik has the form
\eq{1.2c}{S_2 = S_2^0(A,q)-(\sum_{i=1}^N m_i^2A_i)S_1(A)
}
Here $q$ denotes the external momenta, and $S_2^0(A,q)$ is homogeneous of degree $2$ in $q$ and of degree $n+1$ in the $A$. Moreover, $S_2^0$ is a quaternionic pfaffian associated to a quaternionic hermitian matrix, \cite{BK}, so in the case of zero masses there is again the possibility of linking the motive to the geometry of a linear map. In this note we go further and show for the case $N=2n+2$ that $S_2$ is itself a pfaffian via the calculus of {\it twistors}. 

To avoid issues with convergence for the usual propagator integral, I assume in what follows that the masses are positive and the propagators are euclidean. Note that in \eqref{1.4c} the pfaffian can vanish where some of the $a_i=0$. The issues which arise are analogous to issues of divergence already familiar to physicists. They will not be discussed here. 
\begin{thm}\label{main} Let $\Gamma$ be a graph with $n$ loops and $2n+2$ edges as above. We fix masses $m_i>0$ and external momenta $q$ and consider the amplitude
\eq{1.3c}{\sA(\Gamma,q,m)= \int_{\R^{4n}}\frac{d^{4n}x}{\prod_{i=1}^{2n+2}P_i(x,q,m_i)}
}
where the $P_i$ are euclidean. Then there exist alternating bilinear forms $Q_i$ on $\R^{2n+2}$ where $Q_i$ depends on $P_i$, $1\le i\le 2n+2$, and a universal constant $C(n)$ depending only on $n$ such that 
\eq{1.4c}{\sA(\Gamma,q,m)= C(n)\int_\delta \frac{\Omega_{2n+1}}{\text{Pfaffian}(\sum_{i=1}^{2n+2}a_iQ_i)^2}
}
Here $\Omega_{2n+1} = \sum \pm a_i da_1\wedge \cdots \widehat{da_i}\cdots da_{2n+2}$ and $\delta$ is the locus on $\P^{2n+1}$ with coordinate functions $a_i$ where all the $a_i \ge 0$. 
\end{thm}

By way of analogy, the first Symanzik polynomial is given by
\eq{}{S_1(\Gamma)(a_1,\dotsc,a_N) = \det(\sum_{e\ \text{edge}} a_eM_e)
}
where $M_e$ is a rank $1$ symmetric $n\times n$-matrix associated to $(e^\vee)^2$, where $e^\vee: H_1(\Gamma,\R) \to \R$ is the functional which associates to a loop the coefficient of $e$ in that loop. Thus, the amplitude in the case of $n$ loops and $2n$ edges is given by
\eq{}{\sA(\Gamma)= C'(n)\int_\delta \frac{\Omega_{2n-1}}{\det(\sum a_iM_i)^2}
}
where $C'(n)$ is another constant depending only on $n$. 

I want to acknowledge help from S. Agarwala, M. Marcolli, and O. Ceyhan. Much of this work was done during June, 2012 when I was visitng Cal Tech.

\section{Linear Algebra}\label{sect2}

Fix $n\ge 1$ and consider a vector space $V=k^{2n+2} = ke_1\oplus\cdots\oplus ke_{2n+2}$. (Here $k$ is a field of characteristic $0$.) We write $O=ke_1\oplus ke_2$ and $I=ke_3\oplus\cdots\oplus ke_{2n+2}$, so $V = O\oplus I$. $G(2,V)$ will be the Grassmann of $2$-planes in $V$.

We have
\eq{20c}{\text{Hom}_k(O,I)\stackrel{\iota}{ \inj} G(2,V) \stackrel{j}{\inj} \P(\bigwedge^2V).
}
Here $\iota(\psi)=k(e_1+\psi(e_1))\oplus k(e_2+\psi(e_2))$ and $j(W) = \bigwedge^2W \inj \bigwedge^2V$.

Write $V^*$ for the dual vector space with dual basis $e_i^*$. We identify $\bigwedge^2V^*$ with the dual of $\bigwedge^2V$ in the evident way, so $\langle e_i^*\wedge e_j^*,e_i\wedge e_j\rangle = 1$.  For $\alpha\in \bigwedge^2V^*$, the assignment
\eq{21c}{\psi \mapsto \langle(e_1+\psi(e_1))\wedge(e_2+\psi(e_2)), a\rangle
}
defines a quadratic map $q_\alpha: \text{Hom}(O,I) \to k$. 
\begin{lem}\label{lem1c} Assume $0\neq \alpha = v\wedge w$ with $v, w\in V^*$. Then the quadratic map $q_\alpha$ has rank $ 4$.
\end{lem}
\begin{proof} It suffices to show $\langle(\sum x_ie_i)\wedge(\sum y_je_j),v\wedge w\rangle$, viewed as a quadric in the $x_i$ and $y_j$ variables, has rank $4$. By assumption $v, w$ are linearly independent. We can change coordinates so $v=\eps_i^*, w=\eps_j^*$, and $\sum x_ie_i=\sum x'{}_i\eps_i, \sum y_je_j = \sum y'{}_j\eps_j$. The polynomial is then
\eq{}{\langle (\sum x'{}_i\eps_i)\wedge(\sum y'{}_j\eps_j),\eps^*_i\wedge\eps^*_j\rangle = x'{}_iy'{}_j-x'{}_jy'{}_i.
}
This is a quadratic form of rank $4$. 
\end{proof}

Returning to the notation in \eqref{20c}, we can write $I=\bigoplus_{i=1}^n I_i$ with $I_i = ke_{2i+1}\oplus ke_{2i+2}$. We can think of $\text{Hom}(O,I) = \bigoplus \text{Hom}(O,I_i)$ as the decomposition of momentum space into a direct sum of Minkowski spaces. We identify $\text{Hom}(O,I_i)$ with the space of $2\times 2$-matrices, and the propagator with the determinant. With these coordinates, an element in $\text{Hom}(O,I)$ can be written as a direct sum $A_1\oplus\cdots \oplus A_n$ of $2\times 2$-matrices. The propagators have the form $\det(a_1A_1+\cdots a_nA_n)$ with $a_i\in k$. The map $\psi: O \to I$ given by $\psi(e_1) = x_3e_3+\cdots+x_{2n+2}e_{2n+2}$ and $\psi(e_2) = y_3e_3+\cdots+y_{2n+2}e_{2n+2}$ corresponds to the matrices 
\eq{23c}{A_i = \begin{pmatrix}x_{2i+1} & x_{2i+2}\\ y_{2i+1} & y_{2i+2}\end{pmatrix}.
} 

\begin{lem}Let $A_i$ be as in \eqref{23c}. Let 
$$\alpha = (\sum_{i=1}^n a_ie_{2i+1}^*)\wedge (\sum_{i=1}^n a_ie_{2i+2}^*)\in \bigwedge^2V^*.
$$ 
Then the quadratic map $q_\alpha$ in lemma \ref{lem1c} is given by
\eq{}{q_\alpha(A_1\oplus\cdots\oplus A_n) = \det(a_1A_1+\cdots +a_nA_n).
}
\end{lem}
\begin{proof}This amounts to the identity
\ml{}{\det\begin{pmatrix}\sum a_ix_{2i+1} & \sum a_i x_{2i+2} \\ \sum a_i y_{2i+1} & \sum a_i y_{2i+2} \end{pmatrix} = \\
\langle(\sum_{i\ge 3} x_ie_i)\wedge (\sum_{i\ge 3} y_ie_i), (\sum_{i=1}^n a_ie_{2i+1}^*)\wedge (\sum_{i=1}^n a_ie_{2i+2}^*)\rangle.
}
For $i=j$ (resp. $i\neq j$) the coefficient of $a_ia_j$ in this expression is
\ga{}{x_{2i+1}y_{2i+2}-x_{2i+2}y_{2i+1} \\ 
\text{resp.  } x_{2i+1}y_{2j+2} - x_{2i+2}y_{2j+1} + x_{2j+1}y_{2j+2}-x_{2j+2}y_{2i+1}.
}
\end{proof}

The full inhomogeneous propagator, which in physics notation would be written $(p_1,\dotsc,p_n) \mapsto (\sum a_ip_i+s)^2$ with the $p_i$ and $s$ $4$-vectors, becomes in the twistor setup
\ml{28c}{\langle (e_1+\sum_{i\ge 3}x_ie_i)\wedge (e_2+\sum_{i\ge 3} y_ie_i), \\
(c_1e_1^*+c_2e_2^*+\sum_{i\ge 1} a_ie_{2i+1}^*)\wedge(d_1e_1^*+d_2e_2^*+\sum_{i\ge 1} a_ie_{2i+2}^*)\rangle = \\
\det\begin{pmatrix}c_1 & d_1 \\ c_2 & d_2\end{pmatrix} +c_1\sum a_iy_{2i+2} -c_2\sum a_i x_{2i+2} -d_1\sum a_i y_{2i+1} +  \\
d_2\sum a_ix_{2i+1} + 
\det\begin{pmatrix}\sum a_ix_{2i+1} & \sum a_ix_{2i+2} \\ \sum a_iy_{2i+1} & \sum a_iy_{2i+2}\end{pmatrix} = \\
\det\begin{pmatrix}\sum a_ix_{2i+1} + c_1& \sum a_ix_{2i+2} + d_1 \\ \sum a_iy_{2i+1} + c_2 & \sum a_iy_{2i+2}+ d_2 \end{pmatrix}.
}
\begin{remark}\label{rmk} In \eqref{28c}, our $\alpha \in \bigwedge^2V^*$ is of rank $2$, i.e. it is decomposible as a tensor and corresponds to an element in $G(2,V) \subset \P(\bigwedge^2V^*)$, \eqref{20c}. If we want to add mass to our propagator, we simply replace $\alpha$ by $\alpha + m^2e_1^*\wedge e_2^*$, yielding $(\sum a_ip_i + s)^2+m^2$. The massive $\alpha$ represents a point in $\P(\bigwedge^2V^*)$ but not necessarily in $G(2,V^*)$. 
\end{remark}

\section{The Twistor Integral}

In this section we take $k=\C$. 
Consider the maps
\eq{31c}{V\times V-S \xrightarrow{\rho} G(2,V) \xrightarrow{j} \P(\bigwedge^2V). 
}
Here $S= \{(v,w)\ | v\wedge w=0\}$ and $\rho(v,w) = \text{$2$-plane spanned by $v,w$}$. 
\begin{lem}$V\times V-S/G(2,V)$ is the principal $GL_2(\C)$-bundle (frame bundle) associated to the rank $2$ vector bundle $\sW$ on $G(2,V)$ which associates to $g\in G(2,V)$ the corresponding rank $2$ subspace of $V$. 
\end{lem}
\begin{proof}With notation as in \eqref{20c}, let $U=\text{Hom}_\C(O,I)\subset G(2,V)$. We have
\eq{}{\rho^{-1}(U) = \{(z_1,\dotsc,z_{2n+2},v_1,\dotsc,v_{2n+2})\ |\ \det\begin{pmatrix}z_1 & z_2 \\ v_1 & v_2\end{pmatrix}\neq 0\}.
}
We can define a section $s_U: U \to \rho^{-1}(U)$ by associating to $a: O \to I$ its graph
\eq{}{s_U(a) := (1,0,a^1_1,\dotsc,a^1_{2n};0,1,a^2_1,\dotsc,a^2_{2n}).
}
Using this section and the evident action of $GL_2(\C)$ on the fibres of $\rho$, we can identify $\rho^{-1}(U) = GL_2(\C)\times U$. The fibre $\rho^{-1}(u)$ for $w\in U$ is precisely the set of framings $w=\C z\oplus \C v$ as claimed. 
\end{proof}

\begin{lem}The canonical bundle $\omega_{G(2,V)} = \sO(-2n-2)$ where $\sO(-1)$ is the pullback $j^*\sO_{\P(\bigwedge^2V)}(-1)$. 
\end{lem}
\begin{proof}The tautological sequence on $G(2,V)$ reads
\eq{3.4c}{0 \to \sW \to V_{G(2,V)}\to V_{G(2,V)}/\sW \to 0.
}
Here $\sW$ is the rank $2$ sheaf with fibre over a point of $G(2,V)$ being the corresponding $2$-plane in $V$. One has 
\eq{}{\Omega^1_{G(2,V)} = \underline{Hom}(V_{G(2,V)}/\sW,\sW)= (V_{G(2,V)}/\sW)^\vee \otimes \sW. 
}
By definition of the Plucker embedding $j$ above we have $\sO_G(-1) = \bigwedge^2\sW$. The formula for calculating chern classes of a tensor product yields
\eq{}{c_1(\Omega^1_G) = c_1((V_{G(2,V)}/\sW)^\vee)^{\otimes 2}\otimes c_1(\sW)^{\otimes 2n} = \sO_G(-2n-2). 
}
\end{proof}

We now fix a point $a \in \P(\bigwedge^2V^*)$. Upto scale, $a$ determines a non-zero  alternating bilinear form on $V$ which we denote by $Q: (x,y) \mapsto \sum_{\nu,\mu} x_\nu Q^{\nu\mu} y_\mu$. By restriction we may view $Q\in \Gamma(G(2,V), \sO(1))$. By the lemma $\omega_G\otimes\sO(2n+2) \cong \sO_G$, so upto scale there is a canonical meromorphic form $\xi$ on $G(2,V)$ of top degree $4n$ with exactly a pole of order $2n+2$ along $Q=0$. We write 
\eq{}{\xi = \frac{\Xi}{Q^{2n+2}}; \quad 0 \neq \Xi \in \Gamma(G,\omega_G(2n+2)) = \C.
}
\begin{lem}We have 
\eq{}{H^i(V\times V-S,\Q) = \begin{cases} \Q & i= 0, 4n+1, 4n+3, 8n+4 \\
(0) & \text{else}
\end{cases}.
}
\end{lem}
\begin{proof}We compute the dual groups $H^{*}_c(V\times V-S,\Q)$. Note a complex vector space has compactly supported cohomology only in degree twice the dimension. Also, $H^1_c(V-\{0\}) \cong H^0_c(\{0\}) = \Q$. 
Let $p: S \to V$ be projection onto the first factor. The fibre $p^{-1}(v) \cong \C$ for $v\neq 0$ and $p^{-1}(0)=V$. It follows that 
\eq{}{H^i_c(S-\{0\}\times V) \cong H^{i-2}_c(V-\{0\}) = (0);\  i\neq 3, 4n+6.
}
Now the exact sequence
\eq{37c}{H^i_c(S-\{0\}\times V) \to H^{i}_c(S,\Q) \to H^i_c(V,\Q)
}
yields $H^i_c(S) = \Q,\ i=3,4n+4, 4n+6$ and vanishes otherwise. Thus, $H^j_c(V\times V-S) = \Q;\ j= 4,4n+5, 4n+7, 8n+8$ and vanishes otherwise. Dualizing, we get the lemma. 
\end{proof}

Let $R\subset V\times V$ be the zero locus of the alternating form $Q$ on $V$ defined above. Clearly $S\subset R$. 
\begin{lem}\label{lem3.3c} Assume the alternating form $Q$ is non-degenerate. Then we have 
\eq{}{H^{i}(V\times V-R,\Q) = \begin{cases} \Q & i=0, 1, 4n+3, 4n+4 \\
(0) & \text{else}.
\end{cases}
}
\end{lem}
\begin{proof}Again let $p: R \to V$ be projection onto the first factor. We have $p^{-1}(0)=V$ and $p^{-1}(v) \cong \C^{2n+1}$ for $v\neq 0$. It follows that $H^i_c(R-\{0\}\times V)=(0), i\neq 4n+3, 8n+6$. As before, this yields $H^i_c(R) = \Q,\ i=4n+3, 4n+4, 8n+6$ and zero else. Hence $H^j_c(V\times V-R) = \Q,\ j= 4n+4, 4n+5, 8n+7, 8n+8$ and the lemma follows by duality.  
\end{proof}

Note that in the case $n=0,\ \dim V=2$ we have $S=R$ and the two lemmas give the same information, which also describes the cohomology of the fibres of the map $\rho$. Namely, $H^i(\rho^{-1}(pt))=\Q,\ i=0,1,3,4$ and $H^i=(0)$ otherwise.

The form $Q$ induces a quadratic map on $V\times V$ given by $(v,v') \mapsto vQv'$. 
\begin{lem}Choose a basis for $V$ and write $dv$ for the evident holomorphic form of degree $4n+4$ on $V\times V$. Then $\mu := dv/Q^{2n+2}$ is homogeneous of degree $0$ and represents a non-trivial class in $H^{4n+4}_{DR}(V\times V-R)$. 
\end{lem}
\begin{proof}$V\times V-R$ is affine, so we can calculate de Rham cohomology using algebraic forms. There is an evident $\G_m$-action which is trivial on cohomology. Writing a form $\nu$ as a sum of eigenforms for this action, we can assume the $\G_m$-action is trivial on $\nu$, which therefore is written $\nu = Fdv/Q^{2n+2+N}$ for some $N\ge 0$ and $\deg F=2N$.  Since $Q$ is non-degenerate, we can write $F = \sum_i F_i\partial Q/\partial v_i$. Let $(dv)_i$ be the form obtained by contracting $dv$ against $\partial/\partial v_i$. Then
\eq{}{\nu + d\Big(\frac{1}{2n+1+N}\sum F_i(dv)_i/Q^{2n+1+N}\Big) = Gdv/Q^{2n+1+N}. 
}
where $G$ is homogeneous of degree $2(N-1)$. Continuing in this way, we conclude that $\nu$ is cohomologous to a constant times $dv/Q^{2n+2}$. Since by the lemma $H^{4n+4}(V\times V-R) = \Q$, we conclude that $\mu:= dv/Q^{2n+2}$ is not exact. 
\end{proof}

If one keeps track of the Hodge structure, lemma \ref{lem3.3c} can be made more precise. One gets e.g. $H^{4n+4}(V\times V-R,\Q) \cong \Q(-2n-3)$. For a suitable choice of coordinatizations for the two copies of $V$ and a suitable rational scaling for the chain $\sigma$ representing a class in $H_{4n+4}(V\times V-R,\Q)$ we can write the corresponding period as
\eq{}{\int_\sigma d^{2n+2}z\wedge d^{2n+2}v/(\sum z_\mu v_\mu)^{2n+2} = (2\pi i)^{2n+3}.
}
Now we make the change of coordinates $v_\mu = \sum_pQ_\mu^pw_p$ and deduce
\eq{}{\int_\sigma d^{2n+2}z\wedge d^{2n+2}w/(\sum z_\mu Q^{\mu p}w_p)^{2n+2} = \frac{(2\pi i)^{2n+3}}{\det Q}.
}
Here $Q$ is alternating in our case, so $\det Q = \text{Pfaffian}(Q)^2$. 

The ``Feynman trick'' in this context is the integral identity
\eq{}{\frac{1}{\prod_{i=1}^{2n+2}A_i} = (2n+1)!\int_{0^{2n+2}}^{\infty^{2n+2}} \frac{da_1\cdots da_{2n+2}\delta(1-\sum a_i)}{(\sum a_iA_i)^{2n+2}}.
}
We apply the Feynman trick with $A_i = \sum_{\mu, p} z_\mu Q^{\mu p}_{i}w_p$ and integrate over $\sigma$
\ml{3.16c}{\int_\sigma \frac{d^{2n+2}z\wedge d^{2n+2}w}{\prod_{i=1}^{2n+2} (\sum_{\mu, p} z_\mu Q^{\mu p}_{i}w_p)} = \\
(2n+1)!\int_\sigma d^{2n+2}z\wedge d^{2n+2}w\int_{0^{2n+2}}^{\infty^{2n+2}} \frac{da_1\cdots da_{2n+2}\delta(1-\sum a_i)}{(\sum a_i(\sum_{\mu, p} z_\mu Q^{\mu p}_{i}w_p))^{2n+2}} \stackrel{?}{=}\\
(2n+1)!\int_{0^{2n+2}}^{\infty^{2n+2}}da_1\cdots da_{2n+2}\delta(1-\sum a_i) \int_\sigma \frac{d^{2n+2}z\wedge d^{2n+2}w}{(\sum_{\mu, p} z_\mu (\sum a_iQ^{\mu p}_{i})w_p)^{2n+2}} = \\
(2n+1)!(2\pi i)^{2n+3}\int_{0^{2n+2}}^{\infty^{2n+2}}\frac{da_1\cdots da_{2n+2}\delta(1-\sum a_i)}{\text{Pfaffian}(\sum a_i Q_i)^2}.
}
The integral on the right in \eqref{3.16c} can be rewritten as a projective integral as on the right in \eqref{1.4c}:
 \eq{}{\int_{0^{2n+2}}^{\infty^{2n+2}}\frac{da_1\cdots da_{2n+2}\delta(1-\sum a_i)}{\text{Pfaffian}(\sum a_i Q_i)^2} = \int_\delta \frac{\Omega_{2n+1}}{\text{Pfaffian}(\sum_{i=1}^{2n+2}a_iQ_i)^2}.
 }

\section{Proof of theorem \ref{main}}

To finish the proof of theorem \ref{main}, we need to understand the chain of integration $\sigma$ in \eqref{3.16c}. We also need to choose the alternating forms $Q_i$ on the left side of \eqref{3.16c} so the resulting integral coincides upto a constant with the Feynman integral in the statement of the theorem \eqref{1.3c}. 

Put an hermitian metric $||\cdot ||$ on $V$. The induced metric on the bundle of $2$-planes defines a submanifold $M \subset V\times V-S$ where $M$ is the set of pairs $(z,v) \in V\times V-S$ such that $||z||=||v||=1$ and $\langle z,v\rangle = 0$. $M$ is a $\U_2$-bundle which is a reduction of structure of the $GL_2(\C)$ bundle $V\times V-S$. The inclusion $M \subset V\times V-S$ is a homotopy equivalence. In particular, the fibre 
\eq{3.17c}{(R^4\rho_*\Z)_w\cong H^4(M_w) = H^4(\U_2) = \Z\cdot [\U_2].  
}
($\U_2$ is a compact orientable $4$-manifold, so this follows by Poincar\'e duality.) 

For the base, write $G^0:= G(2,V)-\{Q=0\}$ where $Q\in \bigwedge^2V^\vee$ is of rank $2n+2$. $G^0$ is affine (and hence Stein) of dimension $4n$, so $H^i(G^0,\Z)=(0)$ for $i>4n$. Let $\rho^0:V\times V-R \to G^0$ be the $GL_2$ principal bundle obtained by restriction from $\rho$. We are interested in the class in $H^{4n+4}(V\times V-R,\Q)$ (cf. lemma \ref{lem3.3c}) dual to $\sigma$. The grassmann is simply connected, so  by \eqref{3.17c}, necessarily $R^4\rho_*\Z \cong \Z_G$. Since the fibres of $\rho$ have cohomological dimension $4$, we have also 
\eq{}{\Q = H^{4n+4}(V\times V-R,\Q) \cong H^{4n}(G^0,R^4\rho^0_*\Q) \cong H^{4n}(G^0, \Q). 
}
It is not hard to show in fact that $H^{4n}(G^0, \Q) = \Q\cdot c_2(\sW)^n$ where $\sW$ is the tautological rank $2$ bundle on $G(2,V)$ as in \eqref{3.4c}. The interesting question is what if anything this class has to do with the topological closure of real Minkowski space in $G(2,V)$ which is classically the chain of integration for the Feynman integral. 

Recall we have $\Gamma$ a graph with no self-loops and no multiple edges. External edges will play no role in our discussion, so assume $\Gamma$ has none. The chain of integration for the Feynman integral is $\R^{4n}$ where $n$ is the loop number of $\Gamma$. This vector space is canonically identified with $H:= H_1(\Gamma, \R)\otimes_\R \R^4$. In particular, an edge $e\in \text{Edge}(\Gamma)$ yields a functional $e^\vee: H_1(\Gamma,\R) \to \R$ associating to a loop $\ell$ the coefficient of $e$ in $\ell$. 

To avoid divergences, the theorem is formulated for euclidean propagators. Let $q: \R^4 \to \R$ be $q(x_1,\dotsc,x_4)=x_1^2+\cdots +x_4^2$. 
The propagators which appear in the denominator of the integral have the form
\eq{4.3d}{H = H_1(\Gamma, \R)\otimes_\R \R^4\xrightarrow{e^\vee\otimes id_{\R^4}} \R^4 \xrightarrow{q} \R.
}

We take complex coordinates in $\C^4=\R^4\otimes \C$ of the form
\ga{4.4d}{z_1=x_1+ix_2,\ z_2=ix_3+x_4,\ w_1 = ix_3-x_4,\ w_2=x_1-ix_2; \\
 \label{4.5d} x_1=\frac{z_1+w_2}{2},\ x_2=\frac{z_1-w_2}{2i},\ x_3=\frac{z_2+w_1}{2i},\ x_4=\frac{z_2-w_1}{2}.
}
In these coordinates $q=z_1w_2-z_2w_1$ and the real structure is $\R^4 = \{(z_1,z_2,-\overline z_2, \overline z_1)\ |\ z_j \in \C\}$. 

Now take real coordinates for $H_1(\Gamma,\R)$ and let $(z_1^k, z_2^k, w_1^k, w_2^k),\ k\ge 1$ be the resulting coordinates on $H_\C$. It is then the case that for each edge $e$ there are real constants $\alpha_k=\alpha_k(e) \in \R$ not all zero, and the propagator for $e$ is 
\eq{4.6d}{\det \begin{pmatrix} \sum_{k\ge 1} \alpha_k z_1^k & \sum_{k\ge 1} \alpha_kz_2^k \\ -\sum_{k\ge 1} \alpha_k\overline z_2^k & \sum_{k\ge 1} \alpha_k\overline z_1^k \end{pmatrix} = \big|\sum_k \alpha_kz_1^k\big|^2+ \big|\sum_k \alpha_kz_2^k\big|^2.
}
Since the linear functionals associated to the various edges $e$ span the dual space to $H_1(\Gamma,\R)$, we see that a positive linear combination of the propagators is necessarily positive definite on $H_\R$ (i.e. $>0$ except at $0$.) 
Using the cordinates $z_i^k, w_i^k$ we can identify $H_\C$ with an open set in $G=G(2,2n+2)$; namely the point with coordinates $z,w$ is identified with the $2$-plane of row vectors
\eq{3.23c}{\begin{pmatrix}1 & 0 & z_1^1 & z_2^1 & z_1^2 & z_2^2 & \hdots \\
0 & 1 & w_1^1 & w_2^1 & w_1^2 & w_2^2 & \hdots \end{pmatrix}.
}
We throw in two more coordinates $z_1^0, z_2^0$ (resp. $w_1^0, w_2^0$) and view the $z_j^k$ (resp. $w_j^k$) as coordinates of points in $V_\C = \C^{2n+2}$. The fact that the set of non-zero matrices of the form $\begin{pmatrix}z_1 & z_2 \\ -\overline z_2 &  \overline z_1 \end{pmatrix}$ is a group under multiplication means that the set of non-zero $2\times (2n+2)$-matrices
\eq{3.24c}{\begin{pmatrix}z_1^0 & z_2^0 & z_1^1 & z_2^1 & \hdots & z_1^{n} & z_2^n \\ -\overline z_2^0  &  \overline z_1^0 & -\overline z_2^1  &  \overline z_1^1 &  \hdots & -\overline z_2^n  &  \overline z_1^n \end{pmatrix}
}
is closed in $G$. It is clearly the closure in $G$ of the real Minkowski space whose complex points are given in \eqref{3.23c}. It will be convenient to scale the rows by a positive real scalar and assume $\sum_{j,k} |z_j^k|^2=1$, so the resulting locus is compact in $V\times V-R$. We also scale the bottom row by a constant $e^{i\theta}$ of norm $1$. The resulting locus
\ml{3.25c}{\sigma:= \\ 
\Big\{\begin{pmatrix}z_1^0 & z_2^0 & z_1^1 & z_2^1 & \hdots & z_1^{n} & z_2^n \\ -e^{i\theta}\overline z_2^0  &  e^{i\theta}\overline z_1^0 & -e^{i\theta}\overline z_2^1  &  e^{i\theta}\overline z_1^1 &  \hdots & -e^{i\theta}\overline z_2^n  &  e^{i\theta}\overline z_1^n \end{pmatrix}\Big|\ \sum_{j,k} |z_j^k|^2=1 \Big\} \\
 \subset V\times V-R
}
is compact and depends on $4n+4$ real parameters. 

Let $Q_e \in \bigwedge^2 V^\vee$ be the form which associates to \eqref{3.23c} the determinant 
$$\det\begin{pmatrix}\sum_{k\ge 1} \alpha_k(e) z_1^k & \sum_{k\ge 1} \alpha_k(e) z_2^k \\ \sum_{k\ge 1} \alpha_k(e) w_1^k & \sum_{k\ge 1} \alpha_k(e) w_2^k \end{pmatrix}.
$$
Let $a_e>0$ be constants, and let $\widetilde Q = \sum_e a_e Q_e \in \bigwedge^2 V^\vee$. Finally, let $Q_0\in \bigwedge^2 V^\vee$ associate to the matrix \eqref{3.24c} the minor $z_1^0\overline z_1^0+z_2^0\overline z_2^0$. It is clear that $Q:= Q_0 + \widetilde Q$ doesn't vanish on any non-zero matrix of the form \eqref{3.24c}. We conclude:
\begin{prop}Let $G(\R)\subset G$ be the set of points \eqref{3.24c}. Then with $Q$ as above, we have $G(\R)\subset G^0 = G-\{Q=0\}$. 
\end{prop}

The locus $\sigma$, \eqref{3.25c}, projects down to $G(\R)$ with fibre the group $\U_2$. 
\begin{prop}With this choice of $\sigma$ we have
\eq{3.26c}{\int_\sigma \frac{d^{2n+2}z\wedge d^{2n+2}w}{Q^{2n+2}} \neq 0. 
}
\end{prop}
\begin{proof} Let $v_j^{k,\vee}$ be the basis of $V^\vee$ which is dual to the coordinate system $z_j^k$ introduced above. Then one checks that $Q$ as described above is associated to an element 
\eq{}{Q = \sum_{k=0}^{n}b_kv_1^{k,\vee}\wedge v_2^{k,\vee}\in \bigwedge^2 V^\vee;\quad b_k>0. 
}
Applied to the matrix on the right in \eqref{3.25c}, 
\eq{}{Q(\cdots) = e^{i\theta} \sum_{k=0}^n b_k(|z_1^k|^2+|z_2^k|^2)
}
Computing $d^{2n+2}z\wedge d^{2n+2}w$ on the right hand side of \eqref{3.25c} yields 
\ml{}{ie^{(2n+2)i\theta}d\theta\wedge \\
\wedge dz_1^0\wedge\cdots\wedge dz_2^n\wedge\sum_k \Big((\overline z_2^kd\overline z_1^k-\overline z_1^kd\overline z_2^k)\wedge \bigwedge_{j\neq k} (d\overline z_1^j\wedge d\overline z_2^j)\Big).
}
The crucial point is that the $e^{i\theta}$ factor in the integrand \eqref{3.26c} cancels. Rescaling we can reduce to the case where all the $b_k=1$. Integrating over $\sigma$ yields a $2\pi i$ from the $id\theta$ and then an integral over the volume form of the $4n+3$ sphere $\sum_{k=0}^n (|z_1^k|^2+|z_2^k|^2)=1$. This is non-zero. 
\end{proof}

The proof of theorem \ref{main} is now complete. To summarize, given $\Gamma$, one uses the change of coordinates \eqref{4.4d} in order to rewrite the euclidean propagators $P_i$ as determinants of alternating matrices $Q_i$. One uses the discussion in section \ref{sect2}, particularly formula \eqref{28c} and remark \ref{rmk}, to interpret these propagators with external momenta and masses as elements in $\bigwedge^2 V^\vee$, where $V \cong \C^{\text{Edge}(\Gamma)} \cong \C^{2n+2}$. Using \eqref{4.6d}, one sees that a positive linear combination of the $Q_i$ does not vanish on the locus $\sigma$ defined in \eqref{3.25c}. This means that the integrand on the right in \eqref{3.16c} has poles only on the boundary of the chain of integration where some of the $a_i=0$. The integral on the left, given our definition of $\sigma$, is a constant (depending only on $n$) times the euclidean amplitude integral.  

\bibliographystyle{plain}

\renewcommand\refname{References}

\end{document}